\documentclass[11pt]{amsart}

\usepackage{pb-diagram}

\usepackage{amsfonts}
\usepackage{amsmath}
\usepackage{amssymb}

\newcommand{\dbar}{\overline{\partial}}

\newcommand{\ddt}[1]{\frac{\partial #1}{\partial t}}

\newcommand{\OO}{\mathcal{O}}

\newcommand{\ddbar}{\sqrt{-1}\partial\dbar}

\newtheorem{theorem}{Theorem}[section]
\newtheorem{proposition}{Proposition}[section]
\newtheorem{lemma}{Lemma}[section]

\newtheorem{definition}{Definition}[section]
\newtheorem{corollary}{Corollary}[section]

\newcommand{\PP}{\mathbb{P}}
\newcommand{\CC}{\mathbb{C}}
\newcommand{\BB}{\mathcal{B}}

\begin{document}

\title{Some type I solutions of Ricci flow with rotational symmetry}

\author{Jian Song}

\address{Department of Mathematics, Rutgers University, Piscataway, NJ 08854}

\email{jiansong@math.rutgers.edu}

\thanks{Research supported in
part by National Science Foundation grant DMS-0847524 and  a Sloan Foundation Fellowship.}

\begin{abstract} We prove that the Ricci flow on $\CC\PP^n$ blown-up at one point starting with any rotationally symmetric K\"ahler metric must  develop Type I singularities. In particular, if the total volume does not go to zero at the singular time, the parabolic blow-up limit of the Type I Ricci flow along the exceptional divisor is a complete non-flat shrinking gradient K\"ahler-Ricci soliton on a complete K\"ahler manifold homeomorphic to $\CC^n$ blown-up at one point.
\end{abstract}

\maketitle


\section{\bf Introduction}\label{section1}

In this paper, we study the Ricci flow on K\"ahler manifolds defined by
$$X_{n,k} = \PP(\mathcal{O}_{\CC\PP^{n-1}}\oplus \mathcal{O}_{\CC\PP^{n-1}}(-k))$$ for $k, n \in \mathbb{N}^+$. Such manifolds are holomorphic $\CC\PP^1$ bundle over the projective space $\CC\PP^{n-1}$. They are called Hirzebruch surfaces when $n=2$ and $X_{n,1}$ is exactly $\CC\PP^n$ blown-up at one point. The maximal compact subgroup of the automorphism group of $X_{n,k}$ is given by $G_{n,k}=U(n)/ \mathbb{Z}_k$ (\cite{Cal}).

The unnormalized Ricci flow introduced by Hamilton \cite{H1}  is defined on a Riemannian manifold $M$ starting with a Remannian metric $g_0$ by

\begin{equation}\label{rf}
\ddt{g} = -Ric(g), ~~~g(0) = g_0.
\end{equation}

We apply the Ricci flow (\ref{rf}) to $X_{n,k}$ with a $G_{n,k}$-invariant initial K\"ahler metric. In \cite{SW1}, it is shown that the Ricci flow (\ref{rf}) must develop finite time singularity and it either shrinks to a point, collapses to $\CC\PP^{n-1}$ or contracts an exceptional divisor, in Gromov-Hausdorff topology.

When the flow shrinks to a point, $X_{n,k}$ is a Fano manifold and $1\leq k < n$. It is shown by Zhu \cite{Zhu} that the flow must develop Type I singularities and the rescaled Ricci flow converges in Cheeger-Gromov-Hamilton sense to the unique compact K\"ahler-Ricci soliton on $X_{n,k}$ constructed in \cite{Koi, Cao1, WZ}.

When the flow collapses to $\CC\PP^{n-1}$, it is shown by Fong \cite{Fo} that the flow  must develop Type I singularities and the rescaled Ricci flow converges in Cheeger-Gromov-Hamilton sense to the ancient solution that splits isometrically as $\CC^{n-1} \times \CC\PP^1$.

Our main result is to show that the flow must also develop Type I singularities when it does not collapse and the blow-up limit is a nontrivial complete shrinking gradient K\"ahler-Ricci soliton.  Combined with the results of Zhu \cite{Zhu} and Fong \cite{Fo}, we have the following theorem. 

\begin{theorem}\label{main}  Let $X$ be $\CC\PP^n$ blown-up at one point. Then the Ricci flow on $X$ must develop Type I singularities for any $U(n)$-invariant initial K\"ahler metric.

\smallskip

Let $g(t)$ be the smooth solution defined on $t\in [0, T)$, where $T\in (0, \infty)$ is the singular time.  For every $K_j\rightarrow \infty$, we consider the rescaled Ricci flow  $(X, g_j(t'))$ defined on $[-K_j T, 0)$ by
$$g_j(t') = K_j g(T+K_j^{-1} t').$$ Then one and only one of the following must occur.

\begin{enumerate}

\item[(1)] If $\liminf_{t\rightarrow T} (T-t)^{-1}Vol(g(t)) =\infty$, then $(X, g_j(t'), p)$ subconverges in Cheeger-Gromov-Hamilton sense to a  complete  shrinking non-flat gradient  K\"ahler-Ricci soliton on a complete K\"ahler manifold homeomorphic to  $\CC^n$ blown-up at one point, for any $p$ in the exceptional divisor.

\smallskip

\item[(2)] If $\liminf_{t\rightarrow T} (T-t)^{-1}Vol(g(t)) \in (0, \infty)$, then $(X, g_j(t'), p_j)$ subconverges in Cheeger-Gromov-Hamilton sense to $(\CC^{n-1}\times \CC\PP^1, g_{\CC^n} \oplus (-t') g_{FS})$, where $g_{\CC^{n-1}}$ is the standard flat metric on $\CC^{n-1}$ and $g_{FS}$ the Fubini-Study metric on $\CC\PP^1$ for any sequence of points $p_j$  \cite{Fo}.
\smallskip

\item[(3)] If $\liminf_{t\rightarrow T} (T-t)^{-1}  Vol(g(t)) = 0$, then $(X, g_j(t'))$ converges in Cheeger-Gromov-Hamilton sense to the unique compact shrinking K\"ahler-Ricci soliton on $\CC\PP^n$ blown-up at one point \cite{Zhu}.

\end{enumerate}

\end{theorem}

The generalization of Theorem \ref{main} for $X_{n,k}$ is given in section 6. In order to exclude Type II singularities, we first prove a  lower bound for the holomorphic bisectional curvature and then we apply Cao's splitting theorem for the K\"ahler Ricci flow with nonnegative holomorphic bisectional curvature \cite{Cao2}.  Theorem \ref{main} gives evidence that the K\"ahler-Ricci flow can only develop Type I singularities for K\"ahler surfaces and if the flow does not collapse in finite time. Combined with the results of \cite{SW1, SW2}, Theorem \ref{main} verifies that the flow indeed performs a geometric canonical surgery with minimal singularities in the K\"ahler case. We also remark that  the shrinking soliton as the pointed blow-up limit is trivial if the parabolic rescaling takes place at a fixed base point outside the exceptional divisor $D_0$. We believe that the blow-up limit should be the unique  homothetically rotationally symmetric complete shrinking soliton on $\CC^2$ blown-up at one point constructed by Feldman-Ilmanen-Knopf in \cite{FIK}. Unfortunately, we are unable to show that that limiting complete K\"ahler manifold is biholomorphic to $\CC^n$ blown-up at one point, although it has the same topological structure with the unitary group $U(n)$ lying  in the isometry group of the limiting soliton. 

The organization of the paper is as follows. In section 2, we introduce the Calabi ansatz. In section 3, we obtain a  lower bound for the holomorphic bisectional curvature. In section 4, we prove the flow must develop Type I singularities if non-collapsing. In section 5, we construct the blow-up limit. In section 6, we discuss some generalizations of Theorem \ref{main}.

We would also like to mention that we have been informed by Davi Maximo that he has a different approach to understand the singularity formation in similar settings \cite{Ma}.

\bigskip

\section{\bf Calabi symmetry}

In this section, we introduce the Calabi ansatz on $\CC\PP^n$ blown-up at one point introduced by Calabi \cite{Cal} (also see \cite{Cao1, FIK, SW1}). From now on, we let $X$ be $\CC\PP^n$ blown-up at one point and it is in fact a $\CC\PP^1$ bundle over $\CC\PP^{n-1}$ given by
$$X= \PP( \OO_{\CC\PP^{n-1}}\oplus \OO_{\CC\PP^{n-1}}(-1)).$$
Let $D_0$ be the exceptional divisor of $X$ defined by the image of the section $(1,0)$ of $\OO_{\CC\PP^{n-1}}\oplus \OO_{\CC\PP^{n-1}}(-1)$ and $D_\infty$ be the divisor of $X$ defined by the image of the section $(0,1)$ of $\OO_{\CC\PP^{n-1}}\oplus \OO_{\CC\PP^{n-1}}(-1)$. Both the $0$-section $D_0$ and the $\infty$-section are complex hypersurfaces in $X$ isomorphic to $\CC\PP^{n-1}$.  The K\"ahler cone on $X$ is given by
$$\mathcal{K}= \{ -a[D_0] + b [D_\infty] ~|~ 0<a<b\}.$$ In particular, when $n=2$, $D_0$ is a holomorphic $S^2$ with self-intersection number $-1$.

Let $z=(z_1, ..., z_n)$ be the standard holomorphic coordinates on $\CC^n$. Let $\rho = \log |z|^2 = \log (|z_1|^2 + |z_2|^2 + ... + |z_n|^2)$.  We consider a smooth convex function $u=u(\rho)$ for $\rho\in (-\infty, \infty)$ satisfying the following conditions.

\begin{itemize}

\item[(1)]  $u'' >0$ for $\rho\in (-\infty, \infty)$.

\item[(2)] There exist $0<a<b$ and smooth function $u_0, u_\infty: [0, \infty) \rightarrow \mathbb{R}$ such that $$u_0'(0)>0, ~ u_\infty'(0)>0,$$
$$u_0(e^\rho) = u(\rho) - a \rho, ~~~u_\infty(e^{-\rho}) = u(\rho) - b\rho. $$

\end{itemize}

For any $u$ satisfying the above conditions, $\omega=\ddbar u$ defines a smooth K\"ahler metric on $\CC^n\setminus \{0\}$ and it extends to a smooth global K\"ahler metric on $\CC\PP^n$ blown-up at one point in the K\"ahler class $-a[D_0]+b[D_\infty]$.

On $\CC^n\setminus \{0\}$, the K\"ahler metric $g$ induced by $u$ is given by
\begin{equation}
g_{i\bar j} = e^{-\rho} u' \delta_{i\bar j} + e^{-2\rho} \bar z_i z_j (u'' - u').
\end{equation}
Obviously, the K\"ahler metric $g$ induced by $u$ is invariant under the standard unitary $U(n)$ transformations on $\CC^n$.

We define the Ricci potential of $\omega= \ddbar u$ by
 \begin{equation}v=-  \log \det g = n\rho - (n-1) \log u'(\rho) - \log u''(\rho).
\end{equation}
and the Ricci tensor of $g$ is given by
$$R_{i\bar j}  = e^{-\rho} v' \delta_{ij} + e^{-2\rho} \bar z_i z_j (v'' - v').$$

After applying a unitary transformation, we can assume $z=(z_1, 0, ..., 0)$ and then
$$\{ g_{i\bar j} \} = e^{-\rho} diag\{ u'', u', ..., u'\}$$
$$R_{i\bar j}= \sqrt{-1} e^{-\rho} diag\{ v'',  v', ...,  v' \}.$$

The Calabi symmetry is preserved by the Ricci flow, in other words, the evolving K\"ahler metric is invariant under the $U(n)$-action if the Ricci flow starts with a $U(n)$-invariant K\"ahler metric on $X$.

In \cite{SW1}, it is shown that the K\"ahler-Ricci flow on $X$ can be reduced to the following parabolic equation for $u= u(\rho, t)$ for  $\rho\in \mathbf{R}$.

\begin{equation}
\ddt{} u(\rho, t) = \log u''(\rho, t) + (n-1) \log u'(\rho, t) - n\rho +c_t,
\end{equation}
where $$c_t = -\log u''(0, t) - (n-1) u'(0, t)$$ and $u'(\rho, t) = \frac{\partial}{\partial \rho} u(\rho, t).$ The evolving K\"ahler form $\omega(t)$ is then given by $$\omega(t) = \ddbar u(\rho, t).$$

It is also shown in \cite{SW1} that if the initial K\"ahler class is given by $-a_0 [D_0]+b_0[D_\infty]$, the evolving K\"ahler class is given by
$$[\omega(t)] = -a_t [D_0]+b_t [D_\infty], ~~a_t= a_0 -(n-1)t, ~~ b_t = b_0 -(n+1)t.$$
In particular, we have an immediate bound for $u'(\rho, t)$

\begin{equation}
\lim_{\rho\rightarrow-\infty} u'(\rho, t) = a_t, ~~\lim_{\rho\rightarrow \infty} u'(\rho, t) = b_t.
\end{equation}

\section{\bf A lower bound for the holomorphic bisectional curvature}

In this section, we will obtain a lower bound for the holomorphic bisectional curvature. We consider the Ricci flow (\ref{rf}) on $X$ with a $U(n)$-invariant initial K\"ahler metric in the K\"ahler class $-a_0 [D_0]+b_0[D_\infty]$. For our purpose, it suffices to consider the case  $$ 0< a_0 (n+1) < b_0 (n-1).$$ This assumption is shown in \cite{SW1} to be equivalent to the condition
$$\liminf_{t\rightarrow T} Vol(g(t))>0, ~~ \textnormal{or}, ~~\liminf_{t\rightarrow T} (T-t)^{-1} Vol(g(t))= \infty $$
and then the K\"ahler-Ricci flow will contract the exceptional divisor $D_0$ at the singular time $$T= \frac{a_0}{n-1}.$$
We will assume through out this section that the initial K\"ahler class lies in $-a_0[D_0] + b_0 [D_\infty]$ with $ 0< a_0 (n+1) < b_0 (n-1).$

The following theorem is proved in \cite{TZha}.

\begin{theorem}\label{loccin}For any relatively compact set $K$ of $X\setminus D_0$ and $k>0$, there exists $C_{K, k}>0$ such that for all $t\in [0, T)$,
$$|| g(t)||_{C^k(K, g_0)} \leq C_{K, k}.$$

\end{theorem}

It immediately implies that the Ricci flow converges in local $C^\infty$ topology outside the exceptional divisor $D_0$ as $t\rightarrow T$.

The evolution equations for $u', u'', u'''$ are derived in \cite{SW1} as below.

\begin{eqnarray} \label{upevolution}
\ddt{} u' & = & \frac{u'''}{u''} + \frac{(n-1) u''}{u'} -n \\ \label{udpevolution}
\ddt{} u'' & = & \frac{u^{(4)}}{u''} - \frac{(u''')^2}{(u'')^2} + \frac{(n-1)u'''}{u'} - \frac{(n-1) (u'')^2}{(u')^2} \\  \nonumber
\ddt{} u''' & =  & \frac{u^{(5)}}{u''} - \frac{3u''' u^{(4)}}{(u'')^2} + \frac{2(u''')^3}{(u'')^3}+ \frac{(n-1) u^{(4)}}{u'} \\  \label{utpevolution}
&& \mbox{} - \frac{3(n-1) u'' u'''}{(u')^2} + \frac{2(n-1) (u'')^3}{(u')^3}.
\end{eqnarray}

The following lemma is proved in \cite{SW1} for the collapsing case when $a_0 (n+1) > b_0 (n-1) $ and the same proof can be applied here. We include the proof for the sake of completeness.

\begin{lemma} There exists $C>0$ such that for all $t\in [0, T)$ and $\rho \in (-\infty, \infty)$,
\begin{equation}\label{u'}
(n-1) (T-t) \leq  u' \leq C
\end{equation}
and
\begin{equation}\label{u''}
0\leq \frac{u''}{u'} \leq C, ~~-C \leq \frac{u'''}{u''} \leq C.
\end{equation}

\end{lemma}

\begin{proof} The estimate (\ref{u'})  follows from the monotonicity of $u'$  with $a_t < u' \leq b_t$ and $a_t= (n-1)(T-t).$ 

We apply the maximum principle to prove (\ref{u''}).  It is straightforward to verify that for all $t\in [0, T)$,

$$ \lim_{\rho\rightarrow -\infty} \frac{u''(\rho, t)}{u'(\rho, t)}= \lim_{\rho\rightarrow \infty} \frac{u''(\rho, t)}{u'(\rho, t)}=0 $$

$$\lim_{\rho\rightarrow -\infty} \frac{u'''(\rho, t)}{u''(\rho, t)}=1, ~~ \lim_{\rho\rightarrow \infty} \frac{u'''(\rho, t)}{u''(\rho, t)}=-1 . $$

Let $H= \frac{u''}{u'}$. $H$ is strictly positive for all $\rho\in (-\infty, \infty)$ and $t\in [0, T)$. The evolution for $H$ is given by

$$\ddt{ H}= \frac{ H''}{u''} + \frac{2H'}{u'} - \frac{2H^2 - H}{u'} .$$
Therefore $\sup_{\rho\in (-\infty, \infty), t\in [0, T)} H \leq C$ for some uniform constant $C>0$ by applying the maximum principle.

Let $G=\frac{u'''}{u''}$. Then the evolution for $G$ is given by
$$\ddt{} G = \frac{1}{u''} G'' + \left( \frac{n-1}{u'} - \frac{u'''}{(u'')^2} \right) G' - \frac{2(n-1)u''}{(u')^2} \left( G- \frac{u''}{u'} \right).$$
Therefore $\sup_{\rho\in (-\infty, \infty), t\in [0, T)} |G| \leq C$ for some uniform constant $C>0$ by combining the maximum principle and the uniform upper bound for $H$.

\end{proof}

By taking the trace, we obtain an explicit expression for the scalar curvature
\begin{equation}
R = - \frac{ \ddt{u''}}{u''} - \frac{(n-1) \ddt{u'}}{u'}=   -\frac{u^{(4)}}{(u'')^2} + \frac{(u''')^2}{(u'')^3} -  \frac{2(n-1) u'''}{u'u''}  - \frac{(n-1)(n-2) u''}{(u')^2} + \frac{n(n-1)}{u'} .
\end{equation}

\begin{corollary}There exists $C>0$ such that for all $\rho\in (-\infty, \infty)$ and $t\in [0, T)$,
\begin{equation}\label{c4}
-\frac{ u^{(4)}}{ ( u'')^2} + \frac{ (u''')^2 }{ (u'')^3}  \geq - \frac{C}{T-t}.
\end{equation}

\end{corollary}

\begin{proof} Since the scalar curvature $R$ is uniformly bounded below, there exists $C_1>0$ such for all $t\in [0, T)$ and $\rho\in (-\infty, \infty)$,
$$ -\frac{u^{(4)}}{(u'')^2} + \frac{(u''')^2}{(u'')^3} -  \frac{2(n-1) u'''}{u'u''}  - \frac{(n-1)(n-2) u''}{(u')^2} + \frac{n(n-1)}{u'} \geq  - C_1. $$
There also exist $C_2, C_3>0$ such that
$$u' \geq C_2 (T-t)$$
and
$$\left| \frac{u''}{u'} \right| + \left| \frac{ u''' } { u'' } \right| \leq C_3.$$
The estimate (\ref{c4}) immediately follows from the above estimates.

\end{proof}

The holomorphic bisectional curvature $R_{i\bar j k \bar l}$ is computed in \cite{Cao1} and is given by
\begin{eqnarray*}
R_{i\bar j k \bar l} &=& e^{-2\rho} ( u' - u'' ) (\delta_{ij}\delta_{kl} + \delta_{il} \delta_{kj} )\\
&& + e^{-2\rho}  (3u'' - 2 u' - u''') (\delta_{ij} \delta_{kl1} + \delta_{il} \delta_{kj1} + \delta_{kl} \delta_{ij1} + \delta_{kj}\delta_{il1} )\\
&& + e^{-2\rho} \left( 6u''' - 11u'' - u^{(4)} + 6 u' + \frac{(u'' - u''')^2}{u''} \right) \delta_{ijkl1}\\
&& + e^{-2\rho} \frac{ (u' - u'')^2}{ u'} (\delta_{ij\hat 1} \delta_{kl1} + \delta_{il\hat 1} \delta_{kj1} + \delta_{kl\hat 1} \delta_{ij1} + \delta_{kj \hat 1}\delta_{il1} )
\end{eqnarray*}
Here $\delta_{ij1}$ and $\delta_{ijkl1}$ vanish unless all the indices are $1$, while $\delta{ij\hat 1}$ vanishes unless $i=j\neq 1$.

For any point $p$ on $\CC^n \setminus \{ 0\}$, we can assume the coordinates at $p$ are given by  $z(p)=(z_1, ..., z_n) = (z_1, 0, ..., 0)$ after a unitary transformation.

Then all the nonvanishing terms of the holomorphic bisectional curvature are given by
\begin{eqnarray*}
&&R_{1\bar 1 1 \bar 1} = e^{-2\rho} \left( - u^{(4)} + \frac{ ( u''')^2}{u''} \right) \\
&&R_{k\bar k k \bar k} = 2e^{-2\rho} ( u ' - u''), ~k>1 \\
&&R_{1\bar 1 k \bar k} = e^{-2\rho}  \left( - u''' + \frac{ ( u'')^2 }{u'} \right), ~ k>1 \\
&& R_{k \bar k l \bar l}= e^{-2\rho}(  u' - u'' )  ,  ~k>1, l>1, ~k\neq l.
\end{eqnarray*}



\begin{lemma} There exists $C>0$ such that on  for all $t\in [0, T)$, $p = (z_1, 0, ..., 0)$ and $i, j, k, l$, we have at $(p, t)$,
\begin{equation}
R_{i\bar j k \bar l} \geq - \frac{C}{T-t}  (g_{i\bar j} g_{k\bar l} + g_{i\bar l} g_{k \bar j}).
\end{equation}
Furthermore,
\begin{equation}
|R_{i\bar j  k \bar l}| \leq  \frac{C}{T-t}  (g_{i\bar j} g_{k\bar l} + g_{i\bar l} g_{k \bar j})
\end{equation}
unless $i=j=k=l=1$.

\end{lemma}

\begin{proof} Since $p = (z_1, 0, ..., 0)$, it suffices to verify  the estimates for $R_{1\bar 1 1 \bar 1}$, $R_{1\bar 1 k \bar k}$ and $R_{k\bar k l \bar l}$ for $k, l = 2,  ... , n$.

Let $Q_{i\bar j k \bar l} = g_{i\bar j} g_{k\bar l} + g_{i\bar l} g_{k\bar j}.$ Then
\begin{eqnarray*}
&&Q_{1\bar 1 1 \bar 1} = 2e^{-2\rho} (u'')^2 \\
&&Q_{k\bar k k \bar k} = 2e^{-2\rho} ( u')^2, ~k>1 \\
&&Q_{1\bar 1 k \bar k} = e^{-2\rho} u' u'',  ~ k>1 \\
&& Q_{k \bar k l \bar l}= e^{-2\rho}(  u'  )^2  ,  ~k>1, l>1, ~k\neq l.
\end{eqnarray*}
Comparing $R_{i\bar j k\bar l}$ and $Q_{i\bar j k \bar l}$, the lemma follows immediately.

\end{proof}

\begin{proposition}\label{lowbd} The holomorphic bisectional curvature is uniformly bounded below by $-C(T-t)^{-1}$ on $X\times [0, T)$ for some fixed constant $C>0$.
\end{proposition}

\begin{proof}

It suffices to calculate the lower bound of the holomorphic bisectional curvature at a point $p = (z_1, 0, ..., 0)$ and $t\in [0, T)$. Let $V=V^i \frac{\partial}{\partial z_i}$ and $W= W^i \frac{\partial}{\partial z_i}$ be two vectors in $TX_p$.
Then there exists $C>0$ such that
\begin{eqnarray*}
&& R_{i\bar j k\bar l} V^i V^{\bar j} W^k W^{\bar l} \\
&=& R_{1\bar1 1 \bar 1} V^1V^{\bar 1} W^1 W^{\bar 1} + (1- \delta_{ijkl 1} ) R_{i\bar j k \bar l} W^k W^{\bar l} \\
&\geq & - \frac{2C}{T-t}  g_{1\bar 1}g_{1\bar 1} \left|V^1 \right|^2 \left|W^1 \right|^2  - \frac{C}{T-t} (  g_{i\bar j} g_{ k\bar l} + g_{i\bar l} g_{k\bar j} ) \left|V^i \right| \left| V^{\bar j} \right|  \left| W^k \right|  \left| W^{\bar l} \right|\\
&\geq & - \frac{4C}{T-t} \left|V \right|^2_{g} \left|W \right|^2_g.
\end{eqnarray*}

\end{proof}

\begin{definition} Let $g$ be a K\"ahler metric on a K\"ahler manifold $M$. At each point $p\in X$, we can choose the normal coordinates at $p$ such that for $i, j =1 , ..., n$, $g_{i\bar j}(p) =\delta_{i\bar j}$ is the identity matrix and $$R_{i\bar j} (p)= \delta_{ij} \lambda_j.$$  We define the $k^{th}$ symmetric polynomial of Ricci curvature of $g$ at $p$ by
\begin{equation}
\sigma_k = \sigma_k(Ric(g))  = \sum_{j_1< j_2< ... < j_k} \lambda_{j_1} \lambda_{j_2} ... \lambda_{j_k}
\end{equation}  for $1\leq k\leq n$.

\end{definition}

The next proposition gives a uniform bound for $\sigma_k$ in terms of the curvature tensor $R_{1\bar 1 1 \bar 1}$ at each point $z=(z_1, 0, ..., 0)$.
\begin{proposition}  \label{sigmak} There exists $C>0$ such that for all  $(p, t) \in X\times [0, \infty)$,
\begin{equation}
| \sigma_k (p, t) | \leq \frac{C|Rm(p, t)|} { (T-t)^{k -1} }.
\end{equation}

\end{proposition}

\begin{proof} For any point $p\in \CC^n \setminus \{0\}$, we can assume that $p=(z_1, 0, ..., 0)$. Then the eigenvalues of $Ric(g)$ at $p$ with respect to $g$ are given by

\begin{eqnarray*}
&&\lambda_1 = - \frac{\ddt{u''}}{u''}= - \frac{ u^{(4)}} {(u'')^2} + \frac{ (u''')^2}{(u'')^3} - \frac{(n-1) u'''}{u'u''} + \frac{(n-1) u''}{(u')^2}\\
&& \lambda_2 = ... = \lambda_n = -\frac{\ddt{u'}}{u'} = - \frac{u'''}{u'u''} - \frac{ (n-1) u'' }{(u')^2} + \frac{ n }{u'}.
\end{eqnarray*}

Then $(T-t)|\lambda_j|$ is uniformly bounded for $j=2, ..., n$ and
$$|\sigma_k|(p, t) = \sum_{j_1< j_2< ... < j_k } |\lambda_{j_1} \lambda_{j_2} ... \lambda_{j_k}| \leq C (T-t)^{-(k-1)}|\lambda_1| \leq C (T-t)^{-(k-1)} |Rm|_g(p,t).$$

\end{proof}

\begin{lemma}\label{nonfl} For any $p \in D_0$, we have 

\begin{equation}
|Ric(p, t)|_{g(t)} \geq \frac{1}{T-t}.
\end{equation}

\end{lemma}

\begin{proof} It suffices to compute $e^{-\rho} v'$ which is one of the eigenvalues in the Ricci tensors since $D_0 = \{ \rho= -\infty\}$.
\begin{eqnarray*}
\lim_{\rho\rightarrow -\infty} e^{-\rho} v'(\rho)  & =& - \lim_{\rho\rightarrow -\infty}  \frac{ u'''}{u' u''} - \lim_{\rho\rightarrow -\infty} \frac{(n-1) u''}{(u')^2} + \lim_{\rho\rightarrow -\infty} \frac{n}{u'}\\
&=&(n-1) \lim_{\rho\rightarrow -\infty} (u')^{-1}\\
&=&\frac{1}{T-t}.
\end{eqnarray*}
Therefore $|Ric|_g$ is uniformly bounded below by $(T-t)^{-1}$ along the exceptional divisor $D_0$.

\end{proof}

\section{\bf Type I singularities}

In this section, we prove that the Ricci flow must develop Type I singularities with the same assumptions in section 4.

Let's first recall the definition for a Type I singularity of the Ricci flow.
\begin{definition} Let $(M, g(t))$ be a smooth solution of the Ricci flow (\ref{rf}) for $t\in [0, T)$ with $T<\infty$. It is said to develope a Type I singularity at $T$ if it cannot be smoothly extended past $T$ and there exists $C>0$ such that for all $t\in [0, T)$,
\begin{equation}
\sup_{M} |Rm(g(t))|_{g(t)} \leq \frac{C}{T-t}.
\end{equation}

\end{definition}

The following splitting theorem is proved in \cite{Cao2} as a complex analogue of Hamilton's splitting theorem  on Riemannian manifolds with nonnegative curvature operator \cite{H2}.

\begin{theorem}\label{splitting}
Let $g$ be a complete solution of the K\"ahler-Ricci flow
on a noncompact simply connected K\"ahler manifold $M$ of dimension $n$
for $t\in (-\infty, \infty)$ with bounded and nonnegative holomorphic
bisectioanl curvature. Then either $g$ is of positive Ricci curvature for all
$p \in M$ and all $t\in (-\infty, \infty)$, or $(M, g)$  splits holomorphically isometrically into a
product $\CC^k \times N^{n-k}$ $(k \geq 1)$ flat in $\CC^k$ direction and $N$ being of nonnegative holomorphic bisectioanl curvature and positive Ricci curvature.

\end{theorem}

We are now able  to  exclude Type II singularities.

\begin{theorem} \label{type1} Let $X$ be $\CC\PP^n$ blown-up at one point and $g(t)$ be the solution of the K\"ahler-Ricci flow on $X$ starting with a $U(n)$-invariant K\"ahler metric $g_0$. If $g_0$  lies in the K\"ahler class $$-a_0 [D_0] + b_0 [D_\infty]$$ for $ 0< a_0(n+1) < b_0 (n-1)$. Then the flow develops Type I singularities at $T=a_0/(n-1)$.

\end{theorem}

\begin{proof} Suppose the flow develops Type II singularities. Let $t_j$ be an increasing sequence converging to $T=(n-1)a_0>0$  and $p_j$ a sequence of points on $X$ such that
$$K_j = |Rm(p_j, t_j)|_{g(t_j)} =  \sup_{X } |Rm|_{g(t_j)}$$
and
$$\lim_{j\rightarrow \infty} (T-t_j)^{-1} K_j^{-1} = 0. $$

Applying the standard parabolic rescaling, we define
$$ g_j(t) = K_j g(t_j + K_j^{-1} t).$$
After extracting a convergent subsequence, $(X, g_j(t), p_j) $ converges in pointed Cheeger-Gromov-Hamilton sense to a complete eternal solution $(X_\infty, g_\infty(t), p_\infty)$ on a complete K\"ahler manifold $X_\infty$ of dimension $n$.  Furthermore, by the lower bound of the holomorphic bisectional curvature of $g(t)$ by Proposition \ref{lowbd}, the limiting K\"ahler metric $g_\infty(t)$ has nonnegative holomorphic bisectional curvature everywhere on $X_\infty$.  On the other hand, the symmetric product of the Ricci curvature $g_\infty$ vanishes everywhere in $X_\infty$,
$$\sigma_k (Ric(g_\infty)) = 0$$ for $2\leq k \leq n$.
 This implies that the Ricci curvature of $g_\infty$ is not positive at each point of $X_\infty$. By applying the splitting theorem \ref{splitting} for $(n-1)$ times, $(\tilde X_\infty, \tilde g_\infty, \tilde p_\infty)$, the eternal solution on the universal cover of $(X_\infty, g_\infty, p_\infty)$,  splits holomophically isometrically into $\CC^{n-1} \times N$, where $N$ is a compact or complete Riemann surface with positive scalar curvature. By the classification of eternal solutions of real dimension $2$ by  Hamilton \cite{H3}, $(N, \tilde g_\infty(t)|_N )$ is a steady gradient soliton and hence it must be the cigar soliton. However, it violates Peralman's local non-collapsing \cite{P1}, so does $(X_\infty, g_\infty)$. It then leads to a contradiction.

\end{proof}

\section{\bf Blow-up limits}

In this section, we will prove that the blow-up limit of the Ricci flow near the singular time $T$ along the exceptional divisor is a nontrivial complete shrinking gradient K\"ahler-Ricci soliton.

We first prove a diameter bound of the exceptional divisor $D_0$. 

\begin{lemma}  \label{restcp}For all $t \in [0, T)$,
\begin{equation}
g(t)|_{S} =  a_0(n-1)(T-t) g_{FS}.
\end{equation}
and so
\begin{equation}
diam (S, g(t) |_{D_s} ) =   \alpha_n  (a_0 (n-1)(T-t))^{1/2}
\end{equation}
where $g_{FS}$ is a Fubini-Study metric on $\CC\PP^{n-1}$ and $\alpha_n$ is the diameter of $(\CC\PP^{n-1}, g_{FS})$.

\end{lemma}

\begin{proof}
The K\"ahler metric $g(t)$ is the metric completion of the following metric on $\CC^n \setminus \{0\}$
$$ \omega(t) = a_0 (n-1) (T-t) \ddbar \rho + \ddbar u_0(e^\rho, t),$$ where  $u_0(\cdot, t)$ is smooth and for each $t\in [0, T)$ with $u'(0, t)>0$. Note that after extending $\ddbar \rho = \ddbar \log |z|^2 $ to $\CC\PP^n$ blown-up at one point,  its restriction on $D_0$ is exactly a Fubini-Study metric. The lemma then follows immediately.

\end{proof}

Now we can complete the proof of Theorem \ref{main} by identifying the blow-up limit of the Ricci flow at the singular time.

\begin{proposition}\label{blowuplim} Fix any $p\in D_0$. Then for every $K_j \rightarrow \infty$, the rescaled Ricci flows $(X, g_j(t)), p)$ defined on $[-K_j T, 0)$ by
$$g_j (t) = K_j g(T+ K_j^{-1} t)$$
subconverges in Cheeger-Gromov-Hamilton sense to a complete shrinking gradient K\"ahler-Ricci soliton on a complete K\"ahler manifold homeomorphic to $\CC^n$ blown-up at one point.

\end{proposition}

\begin{proof} We first show that the blow-up limit is a nontrivial complete shrinking soliton. Fix any point $p\in D_0$ in the exceptional divisor. Since $(X, g(t))$ is a Type I Ricci flow, the rescaled Ricci flow $(X, g_t(t), p)$ always subconverges to a shrinking gradient soliton $(X_\infty, g_\infty(t), p_\infty)$ in pointed Cheeger-Gromov-Hamilton sense, by the compactness result of Naber \cite{Nab}. Such a limiting soliton cannot be flat because of Lemma \ref{nonfl}. In particular, $(X_\infty, g_\infty, p_\infty)$ is a complete shrinking gradient K\"ahler-Ricci soliton on a complete K\"ahler manifold $X_\infty$.

We now show that $X_\infty$ is in fact homeomorphic to $\CC^n$ blown-up at one point. Fix a closed interval $[a, b] \subset (-\infty, 0)$, the rescaled Ricci flow $g_j (t) $ restricted to $D_0$ is uniformly equivalent to a fixed standard Fubini-Study metric on $\CC\PP^{n-1}$ for all $j$ and $t\in [a, b]$ by Lemma \ref{restcp} and so there exist $d, D>0$ such that  the diameter of $D_0$ with respect to $g_j(t)$ is uniformly bounded between $d$ and $D$ for all $j$ and $t\in [a, b]$.
 We denote by  $$B_{g}(p, R)$$  the geodesic ball with respect to $g$ centered at $p$ with radius $R$. We then consider
$$\BB_{j, t}(D_0, R) = \cup_{p\in D_0} B_{g_j (t)}(p, R)$$
for each $t\in [a, b]$.  By choosing $R$ sufficiently large, we have
$$ B_{g_j(t)} (p, R) \subset   \BB_{j, t}(D_0, R) \subset B_{g_j(t)} (p, 2R) $$ for any point $p\in D_0$ becaue $g_j(t)$ is $U(n)$-invariant.  By definition, for all $t\in [a, b]$, $B_{g_j(t)}(p, R)$ subconverges to $B_{g_\infty(t)} (p_\infty, R)$ in Cheeger-Gromov-Hamilton sense and so $B_{g_\infty(t)}(p_\infty, R)$ is homeomorphic to $B_{g_j(t)}(p, R)$ for sufficiently large $j$. We then obtain an exhaustion $B_{g_\infty (t)} (p, R_k)$ with each $R_k$ sufficiently large and $R_k\rightarrow \infty$. Each of them is homeomorphic to $\CC^n$ blown-up at one point. Therefore $X_\infty$ is homeomorphic to $\CC^n$ blown-up at one point.

\end{proof}
We remark that the convergence in the above proof is $U(n)$-equivariant and the limiting shrinking soliton $(X_\infty, g_\infty, p_\infty)$ is invariant under a free action of the unitary group $U(n)$. We also remark that the Type I blow-up limit is a trivial shrinking soliton if one chooses a fixed base point outside the exceptional divisor $D_0$. This is because the flow converges in local $C^\infty$ topology outside $D_0$ to a smooth K\"ahler metric on $X\setminus D_0$ by Theorem \ref{loccin} \cite{TZha}.

Combing Theorem \ref{type1} and Proposition \ref{blowuplim}, we complete the proof of Theorem \ref{main}.

\section{\bf Some generalizations}

In this section, we discuss some generalizations of Theorem \ref{main}.  First, Theorem \ref{main} can be easily generalized to $X_{n,k}$ defined in section 1 by the same argument in the previous sections.

\begin{theorem}\label{main2} The Ricci flow on $X_{n,k}$ must develop Type I singularities for any  $G_{n,k}$-invariant initial K\"ahler metric.

\smallskip

Let $g(t)$ be the smooth solution defined on $t\in [0, T)$, where $T\in (0, \infty)$ is the singular time.  For every $K_j\rightarrow \infty$, we consider the rescaled Ricci flow  $(X, g_j(t'))$ defined on $[-K_j T, 0)$ by
$$g_j(t') = K_j g(T+K_j^{-1} t').$$ Then one and only one of the following must occur.

\begin{enumerate}

\item[(1)] If $\liminf_{t\rightarrow T} (T-t)^{-1} Vol(g(t)) =\infty$, then $(X, g_j(t'), p)$ subconverges in Cheeger-Gromov-Hamilton sense to a complete nontrivial shrinking  gradient  K\"ahler-Ricci soliton on a complete K\"ahler manifold homeomorphic to the total space of $L^{-k} =\OO_{\CC\PP^{n-1}}(-k)$, for any $p$ in the exceptional divisor.

\smallskip

\item[(2)] If $\liminf_{t\rightarrow T} (T-t)^{-1}Vol(g(t)) \in (0, \infty)$, then $(X, g_j(t'), p_j)$ subconverges in Cheeger-Gromov-Hamilton sense to $(\CC^{n-1}\times \CC\PP^1, g_{\CC^{n-1}} \oplus (-t') g_{FS})$, where $g_{\CC^{n-1}}$ is the standard flat metric on $\CC^{n-1}$ and $g_{FS}$ the Fubini-Study metric on $\CC\PP^1$ for any sequence of points $p_j$  \cite{Fo}.
\smallskip

\item[(3)] If $\liminf_{t\rightarrow T} (T-t)^{-1}  Vol(g(t)) = 0$, then $(X, g_j(t'))$ converges in Cheeger-Gromov-Hamilton sense to the unique compact shrinking K\"ahler-Ricci soliton on $X_{n,k}$ blown-up at one point \cite{WZ}.

\end{enumerate}

\end{theorem}

We can also consider the Calabi symmetry  introduced by Calabi \cite{Cal} for projective bundles over a K\"ahler-Einstein manifold (also see \cite{Li, SY}). In particular, we can consider the Ricci flow on generalizations of $X_{n,k}$
$$X_{m,n, k}= \mathbb{P}(\mathcal{O}_{\CC\PP^n} \oplus \mathcal{O}_{\CC\PP^n}(-k)^{\oplus (m+1)}),~~~k=1, 2, ... .$$
Similar results are obtained for $X_{m, n, k}$ in \cite{SY} for global Gromov-Hausdorff convergence at the singular time, as those for $X_{n,k}$ in \cite{SW1}.
Furthermore, one can obtain the same lower bound for the holomorphic bisectional curvature as in Proposition \ref{lowbd}.

\begin{proposition} Let $g(t)$ be the solution of the Ricci flow on $X_{m,n, k}$ for an initial K\"ahler metric with Calabi symmetry. Then if $1 \leq m \leq n$ and if $$\liminf_{t\rightarrow T} Vol(g(t))>0$$ where $T>0$ is the singular time, then the holomorphic bisectional curvature of $g(t)$ is uniformly bounded below by $-\frac{C}{T-t}$ for some constant $C>0$.

\end{proposition}

Although we are unable to exclude Type II singularities, one can show by the same argument in section 4, that the universal cover of the blow-up limit is an eternal solution of the Ricci flow which splits into $\CC^{n} \times N^{m+1}$ flat in $\CC^n$ and $N^{m+1}$ of nonnegative holomorphic bisectional curvature, if the flow develops Type II singularities. Of course, a Type I bound for the scalar curvature suffices to prove a similar theorem as Theorem \ref{main}.

\bigskip
\bigskip

\noindent{\bf Acknowledgements} The author would like to thank Zhenlei Zhang for many stimulating discussions. He would like also like to thank Valentino Tosatti for many helpful suggestions.


\bigskip
\bigskip



\begin{thebibliography}{99}

\bibitem{A} Aubin, T.  {\em \'Equations du type Monge-Amp\`ere sur les vari\'et\'es k\"ahl\'eriennes compactes},  Bull. Sci. Math. (2) {\bf 102} (1978), no. 1, 63--95

\bibitem{Cal} Calabi, E. {\em  Extremal K\"ahler metrics}, in Seminar on Differential Geometry,  pp. 259--290, Ann. of Math. Stud., {\bf 102}, Princeton Univ. Press, Princeton, N.J., 1982


\bibitem{Cao1} Cao, H.-D. {\em Existence of gradient K\"ahler-Ricci solitons},  Elliptic and parabolic methods in geometry (Minneapolis, MN, 1994), 1--16, A K Peters, Wellesley, MA, 1996

\bibitem{Cao2} Cao, H.-D. {\em   dimension reduction in the K\"ahler-Ricci flow}, Comm. Anal. Geom. 12 (2004), no. 1-2, 305--320

\bibitem {Ch} Chow, B. {\em The Ricci flow on the 2-sphere},
J. Differential Geom. {\bf 33} (1991) 325--334

\bibitem{FIK} Feldman, M., Ilmanen, T. and Knopf, D. {\em Rotationally symmetric shrinking and expanding gradient K\"ahler-Ricci solitons},  J. Differential Geom.  {\bf 65}  (2003),  no. 2, 169--209

\bibitem{Fo} Fong, T. {On the collapsing rate of the K\"ahler-Ricci flow with finite-time singularity}, arXiv:1112.5987


\bibitem{Koi} Koiso, N. {\em On rotationally symmetric Hamilton's equation for K\"ahler-Einstein metrics}, Recent topics in differential and analytic geometry, 327--337, Adv. Stud. Pure Math., 18-{\rm I}, Academic Press, Boston, MA, 1990


\bibitem{H1}Hamilton, R. S. {\em Three-manifolds with positive Ricci curvature}, J. Differ. Geom. 17 (1982), no. 2, 255--306


\bibitem{H2} Hamilton, R. S. {\em Four-manifolds with positive curvature operator}, J. Differ. Geom. 24 (1986), 153--179

 \bibitem{H3} Hamilton, R.S. {\em The formation of singularities in the Ricci flow}, Surveys in differential geometry, Vol. II (Cambridge, MA, 1993), 7--136, Int. Press, Cambridge, MA, 1995




\bibitem{Li} Li, C. {\em  On rotationally symmetric K\"ahler-Ricci solitons}, arXiv:1004.4049

\bibitem{Ma} Maximo, D. {\em On the blow-up of four dimensional Ricci flow singularities}, preprint

\bibitem{Nab} Naber, A. {\em  Noncompact shrinking 4-solitons with nonnegative curvature}, J. Reine Angew. Math. 645 (2010), 125--153

\bibitem{P1} Perelman, G. {\em The entropy formula for the Ricci flow and its geometric applications},  arXiv:math.DG/0211159

\bibitem{P2} Perelman, G. unpublished work on the K\"ahler-Ricci flow

\bibitem{PS} Phong, D.H. and Sturm, J.  {\em On stability and the convergence of the K\"ahler-Ricci flow},  J. Differential Geom.  {\bf 72} (2006),  no. 1, 149--168


 \bibitem{SW1} Song, J. and Weinkove, B. {\em The K\"ahler-Ricci flow on Hirzebruch surfaces}, J. Reine Angew. Math. 659 (2011), 141--168

\bibitem{SW2} Song, J. and Weinkove, B. {\em Contracting exceptional divisors by the K\"ahler-Ricci flow}, arXiv:1003.0718



\bibitem{SY} Song, J. and Yuan, Y. {\em  Metric flips with Calabi ansatz},  to appear in G.A.F.A., arXiv:1011.1608

\bibitem{T} Tian, G. {\em K\"ahler-Einstein metrics with positive scalar curvature}, Invent. Math. {\bf 130} (1997), no. 1, 1--37

\bibitem{TZha} Tian, G. and Zhang, Z. {\em On the K\"ahler-Ricci flow on projective manifolds of general type},  Chinese Ann. Math. Ser. B  {\bf 27}  (2006),  no. 2, 179--192

\bibitem{Ts} Tsuji, H. {\em Existence and degeneration of K\"ahler-Einstein metrics on minimal algebraic varieties of general type}, Math.
Ann. {\bf 281} (1988), 123--133

\bibitem {WZ} Wang, X.J. and Zhu, X.
{\em K\"ahler-Ricci solitons on toric manifolds with positive first Chern class},
Advances Math. {\bf 188} (2004) 87--103

\bibitem{Y1} Yau, S.-T. {\em On the Ricci curvature of a compact K\"ahler
manifold and the complex Monge-Amp\`ere equation, I}, Comm. Pure
Appl. Math. {\bf 31} (1978), 339--411

\bibitem{Y2} Yau, S.-T. {\em Open problems in geometry}, Proc. Symposia Pure
Math. {\bf 54} (1993), 1--28 (problem 65)

\bibitem{Zha} Zhang, Z. {\em On degenerate Monge-Amp\`ere equations over closed K\"ahler manifolds},   Int. Math. Res. Not. {\bf 2006}, Art. ID 63640, 18 pp

\bibitem{Zhu} Zhu, X. {\em K\"ahler-Ricci flow on a toric manifold with positive first Chern class}, arXiv:math/0703486



 \end{thebibliography}
\end{document}